\documentclass{article}

\usepackage{graphicx}
\usepackage[T1]{fontenc}
\usepackage[latin1]{inputenc}
\usepackage{array}
\usepackage{color}

\usepackage{amsmath,amsfonts,amssymb}
\usepackage{amsthm}
\usepackage{amsfonts}
\usepackage{stmaryrd}
\usepackage{ulem}

\let\ds=\displaystyle

\let\to=\rightarrow
\newcommand{\be}{\begin{enumerate}}
\newcommand{\ee}{\end{enumerate}}
\newcommand{\bi}{\begin{itemize}}
\newcommand{\ei}{\end{itemize}}

\def \pa#1{{\left(#1\right)}}

\def\bs{\bigskip}
\def\ms{\medskip}

\font\ineg=msam8
\def\ie{\mathrel{\hbox{\ineg 6}}} 
\def\se{\mathrel{\hbox{\ineg >}}} 

\font\matcinq=msbm5
\font\matsept=msbm7
\font\matdix=msbm10
\newfam\matfam \scriptscriptfont\matfam=\matcinq \scriptfont\matfam=\matsept \textfont\matfam=\matdix
\def\mat{\fam\matfam}
\def\N{{\mat N}}

\newcommand{\bib}[2]{\hbox{\hbox to 12mm{[#1] :\hfill} \hfill \hbox to 144mm{\vtop{\hsize=144mm#2\vfill\ms}\hfill}}}

\let\phi=\varphi

\newtheorem{Thm}{Theorem I-\!}[section]
\newtheorem{Lem}{Lemma I-\!}[section]

\def\wpa{\varcurlyvee} 
\def\twpa{\widetilde{\varcurlyvee}} 

\begin{document}

\begin{center}
\Large{\textsc{Structure and bases of modular space sequences}}

\Large{\textsc{$(M_{2k}(\Gamma_0(N)))_{k\in \mathbb{N}^*}$ and $(S_{2k}(\Gamma_0(N)))_{k\in \mathbb{N}^*}$}}
\bs
\ms

\large\textsc{Part III: Cuspidal spaces}
\end{center}
\ms

\begin{center}
\large Jean-Christophe Feauveau
\footnote{Jean-Christophe Feauveau,\\
Professeur en classes préparatoires au lycée Bellevue,\\
135, route de Narbonne BP. 44370, 31031 Toulouse Cedex 4, France,\\
email: Jean-Christophe.Feauveau@ac-toulouse.fr
}

\end{center}

\begin{center}
\large September 01, 2018
\end{center}

\bs

\textsc{Abstract.}

Based on the notion of strong modular form developed in Part I \cite{FeauFM1}, we propose to structure the family of cuspidal  modular form spaces $(S_{2k}(\Gamma_0(N)))_{k\in \; \mathbb{N}^*}$ and to determine bases for each of these spaces, once known bases for the first values of $k$.

We then apply these theoretical results to explicitly determine bases for space families $(S_{2k}(\Gamma_0(N)))_{k\in \; \mathbb{N}^*}$ when $1\leq N \leq 10$.
\bs
\bs

\bs

Classification A.M.S. 2010: 11F11, 11G16, 11F33, 33E05.
\bs
\bs


\section{-- Structure and bases of $\pa{S_{2k}(\Gamma(N))}_{k\in \N^*}$}
\ms

The purpose of this paragraph is to structure the cuspidal modular form spaces of a given level $N$, and to obtain exploitable results to build unitary upper triangular bases for these spaces $\pa{S_{2k}(\Gamma(N))}_{k\in \N^*}$.
\ms

We write $\delta_{2k}(N) = \dim(S_{2k}(\Gamma_0(N)))$. 
Thereafter, ${\cal B}_{2k}(\Gamma_0(N)) = (E_{2k,N}^{(s)})_{0\ie s \ie d_{2k}(N)-1}$ denotes a $M_{2k}(\Gamma_0(N))$ unitary upper triangular basis while ${\cal C}_{2k}(\Gamma_0(N)) = (F_{2k,N}^{(s)})_{1\ie s\ie \delta_{2k}(N)}$ refers to a $S_{2k}(\Gamma_0(N))$ unitary upper triangular basis.
\ms

Knowing a family $({\cal B}_{2k}(\Gamma_0(N)))_{k\in\N^*}$, we are looking for a family $({\cal C}_{2k}(\Gamma_0(N)))_{k\in\N^*}$. For that, the main tool remains the strong $\Delta_N$ modular form, even if this one is not cupsidal.
\ms

The application $\varphi: S_{2k}(\Gamma_0(N)) \to \{\Phi\in S_{2k+\rho_N}(\Gamma_0(N)) \ / \nu(\Phi) > \nu(\Delta)\}$ defined by $\varphi(\Phi) = \Delta_N\Phi$ is clearly an isomorphism, and therefore:
\begin{equation}
\forall k\in \N^*, \ \ \Delta_N S_{2k}(\Gamma_0(N)) = \{\Phi\in S_{2k+\rho_N}(\Gamma_0(N)) \ / \ \nu(\Phi) > \nu(\Delta)\}.
\end{equation}

\ms

The isomorphism $\varphi$ leads to the following result.
\begin{Lem}\label{DimS2}
With previous notations, for $2k \se 2$,
\begin{equation}
S_{2k+\rho_N}(\Gamma_0(N)) = Vect(F_{2k+\rho_N,N}^{(s)} \ / \ \nu(F_{2k+\rho_N,N}^{(s)}) \ie \nu(\Delta_N)) \oplus \Delta_N S_{2k}(\Gamma_0(N)).
\end{equation}

In addition, you can choose
\begin{equation}
\forall s \in \llbracket\nu(\Delta_N)+1, \delta_{2k}(N)\rrbracket, \ \ F_{2k+\rho_N,N}^{(s)} = \Delta_N.F_{2k,N}^{(s-\nu(\Delta_N))}.
\end{equation}

\end{Lem}
\ms

The formulae giving the dimensions of the modular spaces and Theorem I-7.1 give the following result.

\begin{Lem}\label{LemDim3}
Let $N$ be in $\N^*$, we write down $\rho_N$ the $\Delta_N$ weight. For $k \se 2$,
\begin{equation}
\dim(S_{2k+\rho_N}(\Gamma_0(N))) - \dim(S_{2k}(\Gamma_0(N))) = \nu(\Delta_N). 
\end{equation}
and for $k=1$,
\begin{equation}
\dim(S_{2+\rho_N}(\Gamma_0(N))) - \dim(S_{2}(\Gamma_0(N))) = \nu(\Delta_N)-1. 
\end{equation}
\end{Lem}
\ms

But from lemma III-\ref{DimS2} we deduce equality

\begin{equation}
\forall k\se 1, \ \ \dim(S_{2k+\rho_N}(\Gamma_0(N))) = \dim(S_{2k}(\Gamma_0(N))) + Card(\{s \ / \ \nu(F_{2k+\rho_N,N}^{(s)}) \ie \nu(\Delta_N)\})
\end{equation}

Therefore, if $k \se 2$,
\begin{equation}
Card(\{s \ / \ \nu(F_{2k+\rho_N,N}^{(s)}) \ie \nu(\Delta_N)\}) = \nu(\Delta_N)
\end{equation}
and
\begin{equation}
Card(\{s \ / \ \nu(F_{2+\rho_N,N}^{(s)}) \ie \nu(\Delta_N)\}) = \nu(\Delta_N)-1.
\end{equation}
\ms

From these results on the valuations of the elements of the unitary upper triangular bases, we derive the following theorem.

\begin{Thm}\label{DimS3}
Let be $N\in\N^*$ and, for the whole $k\se 1$, $(F_{2k,N}^{(s)})_{1\ie s \ie \delta_{2k}(N)}$ a $S_{2k}(\Gamma_0(N))$ unitary upper triangular basis. So
\begin{equation}
\forall k \se \frac{\rho_N}{2}+2, \ \ \forall s\in \llbracket1,\nu(\Delta_N)\rrbracket, \ \ \nu(F_{2k,N}^{(s)}) = s.
\end{equation}
In addition, you can choose
\begin{equation}
\forall k \se \frac{\rho_N}{2}+2, \ \ \forall s\in \llbracket1,\nu(\Delta_N)\rrbracket, \ \ F_{2k,N}^{(s)} = F_{\rho_N+4,N}^{(s)}[E_{2,N}^{(0)}]^{k-\frac{\rho_N}{2}-2},
\end{equation}
where $E_{2,N}^{(0)}$ is an element of $M_{2}(\Gamma_0(N))$ unitary and zero valuation.
\end{Thm}
\bs

The following theorem, a direct consequence of lemma III-\ref{DimS2} and Theorem III-\ref{DimS3}, enables one to structure the family $(S_{2k}(\Gamma_0(N)))_{k\in \N^*}$ and to describe the construction of unitary upper triangular bases.

\begin{Thm}\label{ThmStruct102}
Let $N$ be a positive integer, then
\begin{equation}
\forall k\in \N, \ k\se \frac{\rho_N}{2} +2, \ \ S_{2k}(\Gamma_0(N)) = \Delta_N.S_{2k-\rho_N}(\Gamma_0(N)) \oplus Vect\pa{F_{\rho_N+4,N}^{(s)}[E_{2,N}^{(0)}]^{k-\frac{\rho_N}{2}-2} \ / \ 1\ie s \ie \nu(\Delta_N)}.
\end{equation}
Therefore, if $k\se 2$ and $k = q\frac{\rho_N}{2}+r$ with $2\ie r \ie \frac{\rho_N}{2}+1$,
\begin{equation}
S_{2k}(\Gamma_0(N)) = \Delta_N^{q}. S_{2r}(\Gamma_0(N)) \bigoplus_{n=0}^{q-1}  \Delta_N^n .Vect\pa{F_{\rho_N+4,N}^{(s)}[E_{2,N}^{(0)}]^{k-(n+1)\frac{\rho_N}{2}-2} \ / \ 1\ie s \ie \nu(\Delta_N)}.
\end{equation}
\end{Thm}

So, the modular forms $\Delta_N$ and $E_{2,N}^{(0)}$ being known \cite{FeauFM1,FeauFM2}, the knowledge
\be
\item[(i)]
of bases $({\cal C}_{2k}(\Gamma_0(N)))_{1\ie k\ie \rho_N+1}$
\item[(ii)]
of the first $\nu(\Delta_N)$ items of ${\cal C}_{\rho_N+4}(\Gamma_0(N))$
\ee
provides a set of $({\cal C}_{2k}(\Gamma_0(N)))_{k\in \N^*}$ unitary upper triangular bases adapted to $(S_{2k}(\Gamma_0(N)))_{k\in \N^*}$.
\bs

\bs
\bs

\section{ -- From theory to practice}
\ms

Remember that ${\cal B}_{2k}(\Gamma_0(N)) = (E_{2k,N}^{(s)})_{0\ie s \ie d_{2k}(\Gamma_0(N))-1}$ denotes a $M_{2k}(\Gamma_0(N))$ unitary upper triangular basis while ${\cal C}_{2k}(\Gamma_0(N)) = (F_{2k,N}^{(s)})_{1\ie s \ie \delta_{2k}(\Gamma_0(N))}$ refers to a $S_{2k}(\Gamma_0(N))$ unitary upper triangular basis we are looking for.
\ms

Like Theorem I-7.3 for spaces $(M_{2k}(\Gamma_0(N)))_{k\in \N^*}$, Theorem III-\ref{ThmStruct102} enables one recursively or explicitly determining bases of $(S_{2k}(\Gamma_0(N)))_{k\in \N^*}$. Instead of redoing Part II work to get the bases $({\cal B}_{2k}(\Gamma_0(N)))_{n\in \N^*}$, we will rely on knowledge of these bases to determine bases $({\cal C}_{2k}(\Gamma_0(N)))_{n\in \N^*}$ when $1\ie N \ie 10$. This is the subject of the following result.
\ms



\begin{Lem}\label{LemP}
Let $k_0\in \N^*$ and $(F_{2k_0,N}^{(s)})_{1\ie s\ie \delta_{k_0}(N)}$ a $S_{2k_0}(\Gamma_0(N))$ unitary upper triangular basis. So
\be
\item[(i)]
For $1 \ie s \ie \delta_{k_0}(N)$, and an integer $k\se k_0$, $F_{2k_0,N}^{(s)}.{\cal B}_{2(k-k_0)}(\Gamma_0(N))$ is a linearly independent and unitary upper triangular family of $S_{2k}(\Gamma_0(N))$.
\item[(ii)]
For an integer $k \se k_0$, if $\dim(S_{2k}(\Gamma_0(N))) = \dim(M_{2(k-k_0)}(\Gamma_0(N))) + \dim(S_{2k_0}(\Gamma_0(N)))-1$, then $\pa{F_{2k_0,N}^{(u)}[E_{2,N}^{(0)}]^{k-k_0}, \ 1\ie u \ie \delta_{k_0}(N)-1}\cup F_{2k_0,N}^{(\delta_{k_0}(N))}.{\cal B}_{2(k-k_0)}(\Gamma_0(N))$ is a $S_{2k}(\Gamma_0(N))$ unitary upper triangular basis.

\item[(iii)]
If there is $k_0 \se 1$ for which the condition $(ii)$ applies to $k\in \{k_0+i, \ 1\ie i \ie 6\}$, then the result $(ii)$ is true for all $k \se k_0$.
We can then choose for any $k\se k_0$:

${\cal C}_{2k}(\Gamma_0(N)) = \pa{F_{2k_0,N}^{(u)}[E_{2,N}^{(0)}]^{k-k_0}, \ 1\ie u \ie \delta_{k_0}(N)-1}\cup F_{2k_0,N}^{(\delta_{k_0}(N))}.{\cal B}_{2(k-k_0)}(\Gamma_0(N))$.
\ee
\end{Lem}

\begin{proof} $ $

\be
\item[(i)]
The result comes directly from the properties of ${\cal B}_{2(k-k_0)}(\Gamma_0(N))$ basis.

\item[(ii)]
Family $\pa{F_{2k_0,N}^{(u)}[E_{2,N}^{(0)}]^{k-k_0}, \ 1\ie u \ie \delta_{k_0}(N)-1}\cup F_{2k_0,N}^{(\delta_0(N))}.{\cal B}_{2(k-k_0)}(\Gamma_0(N))$ is linearly independent and unitary upper triangular in $S_{2k}(\Gamma_0(N))$. It also contains exactly $(S_{2k}(\Gamma_0(N)))$ elements, so it is a space basis.

\item[(iii)]
For $k > k_0$ the $\varphi: k\mapsto \dim(S_{2k}(\Gamma_0(N))) - \dim(M_{2(k-k_0)}(\Gamma_0(N)))$ is $6$-periodic, it is a direct consequence of the formulae giving the modular space dimensions \cite{Stein}. The assumption $(iii)$ requires this function to be constant, starting from $k = k_0+1$. We can then apply $(ii)$ to any $k\se k_0$.
\ee
\end{proof}
\ms
We will systematically use this result for $1 \ie N \ie 10$. For $N$ other than $7$ and $10$, it will come $\delta_{k_0}(N) = 1$ whereas for $N = 7$ and $N = 10$, we will have $\delta_{k_0}(N) = 3$. Note that there are values of $N$ for which the situation $(iii)$ does not occur for any value of $k_0$, $N = 26$ for example.
\ms

Verification of the condition on dimensions indicated in $(ii)$ is possible thanks to the space dimension formulae reminded in Part I, and for the following paragraph, using SAGE.
\ms

Other approaches are possible. By noting, for example, that the function $\Delta(N\tau)$ belongs to $S_{12}(\Gamma_0(N))$. We will not use this remark in the following.
\bs
\bs

\section{ -- Structure and bases of $(S_{2k}(\Gamma_0(N)))_{k\in \N^*}$, $1\ie N \ie 10$}
\ms

For a given level $1\ie N \ie 10$, the application of Theorem III-\ref{ThmStruct102} is based on knowledge of the $((F_{2k,N}^{(s)})_{1\ie s \ie \delta(N)})$ for ${1\ie k\ie \frac{1}{2}\rho_N+1}$ and $(F_{\rho_N+4,N}{(s)})_{1\ie s \ie \nu(\Delta_N)}$.
\ms

The determination of these bases will be based on lemma III-\ref{LemP}. When cupsidal forms are missing to complete a base, we will remember that the space $S_{2k}(\Gamma_0(N))$ is included in $M_{2k}(\Gamma_0(N))$ of which we know a basis from \cite{FeauFM2}. An effective strategy will be to remind the bases of $M_{2k}(\Gamma_0(N)$ obtained for the first $k$ values then, thanks to a asymptotic expansion obtained by SAGE, we express the elements $F_{2k,N}^{(s)}$ searched in the known ${\cal B}_{2k}(\Gamma_0(N)) = (E_{2k,N}^{(s)})_{0\ie s \ie d_{2k}(\Gamma_0(N))-1}$ basis.
\bs

As was the case in Part II, we will try to produce bases expressing itself according to the function of Dedekind $\eta$ and, as often as possible, according to the elliptic functions $\wpa$ and $\twpa$, see \cite{FeauFM2}. The reference variable $\tau$ belongs to the Poincaré half plane ${\cal H}$ and we put $q = e^{2i\pi\tau}$.

\subsection{ -- Structure and bases of $(S_{2k}(\Gamma_0(1)))_{k\in \N^*}$}
\ms

The strong modular unit that structures this set of spaces is $\Delta_1 = \Delta$.

Let us remind the first values of $\delta_{2k}(\Gamma_0(1))$:

\begin{center}
\begin{tabular}{c|c|c|c|c|c|c|c|c|c}
$2k$   & 2  & 4  & 6  & 8  & 10  & 12  & 14  & 16 & 18\\
\hline
$\delta_{2k}(\Gamma_0(1))$   & 0  & 0  & 0  & 0  & 0  & 1  & 0  & 1 & 1 \\
\end{tabular}
\end{center}

As a reminder, you can choose
\begin{equation*}
\left\{
\begin{array}{lcll}
\ds E_{2k,1}^{(0)} & = & E_4^{k/2} & \text{si} \ k\in 2\N^*\\
\ds E_{2k,1}^{(0)} & = & E_4^{(k-3)/2} E_6 & \text{si} \ k\in 2\N^*+1.
\end{array}
\right.
\end{equation*}
\ms

Remember that Eisenstein functions $E_4$ and $E_6$ are written according to $\twpa$.
\ms

For $k = 6q+r\se 2$, $1\ie r \ie 6$, a $M_{2k}(\Gamma_0(1))$ unitary upper triangular basis ${\cal B}_{2k}(\Gamma_0(1))$ is given by

\begin{equation*}
\left\{
\begin{array}{ll}
\ds (\Delta^n E_{2k-12n,1}^{(0)})_{0\ie n\ie q-1} & \text{if} \ r=1\\
\ds (\Delta^n E_{2k-12n,1}^{(0)})_{0\ie n\ie q} & \text{if} \ 2\ie r \ie 5\\
\ds (\Delta^{q+1})\cup(\Delta^n E_{2k-12n,1}^{(0)})_{0\ie n\ie q} & \text{if} \ r=6.
\end{array}
\right.
\end{equation*}
This basis is unitary, unitary upper triangular without jump and is expressed in terms of $\twpa$.
\bs

The generator for $S_{12}(\Gamma_0(1))$ is:
\begin{equation}
F_{12,1}^{(1)} = \Delta(\tau) = q\prod _{k=1}^{+\infty}(1-{q}^{k})^{24} = \frac{1}{256} \pa{\twpa\pa{\frac{1}{2},\tau}\twpa\pa{\frac{\tau}{2},\tau}\twpa\pa{\frac{\tau+1}{2},\tau+1}}^2.
\end{equation}

This result is obtained, for example, with the factorization of $\twpa$. We can also notice that the right member is in $M_{12}(\Gamma_0(1))$ and the first two terms of his asymptotic expansion coincide with those of $\Delta$.
\ms

For $k\in \llbracket 7,12\rrbracket$, $\dim(S_{2k}(\Gamma_0(1))) - \dim(M_{2(k-6)}(\Gamma_0(1))) = 0 = \dim(S_{12}(\Gamma_0(1)))-1$.

In application of lemma III-\ref{LemP}, for $k_0 = $6, the generic form of a $S_{2k}(\Gamma_0(1))$ unitary upper triangular basis ${\cal C}_{2k}(\Gamma_0(1))$ is:
\begin{equation*}
\forall k \se 6, \ \ {\cal C}_{2k}(\Gamma_0(1)) = F_{12,1}^{(1)}{\cal B}_{2k-12}(\Gamma_0(1)).
\end{equation*}
\ms
\bs

\subsection{-- Structure and bases of $(S_{2k}(\Gamma_0(2)))_{k\in \N^*}$}
\ms

The modular form that structures this set of spaces is defined as follows:

\begin{equation}
\Delta_2(\tau) =  \eta(2\tau)^{16} \eta(\tau)^{-8} = \frac{1}{256} {\twpa}(\frac{1}{2},\tau)^2 =  q+8q^2+28q^3+64q^4+O(q^5).
\end{equation}

Let us remind the first values of $\delta_{2k}(\Gamma_0(2))$:

\begin{center}
\begin{tabular}{ c|c|c|c|c|c|c|c|c|c|c}
$2k$   & 2  & 4  & 6  & 8  & 10  & 12  & 14  & 16 & 18\\
\hline
$\delta_{2k}(\Gamma_0(2))$ & 0  & 0  & 0  & 1  & 1   &  2  & 2   & 3 & 3\\
\end{tabular}
\end{center}

We determined a $M_{2}(\Gamma_0(2))$ unitary upper triangular basis: $\ds E_{2,2}^{(0)} = -3\wpa(\tau,2\tau)$.
\ms

We have a $M_{4}(\Gamma_0(2))$ unitary upper triangular basis:
\ms

$\ds E_{4,2}^{(0)}(\tau) = [E_{4,2}^{(0)}(\tau)]^2 = 9 {\wpa}(\tau,2\tau)^2 = 1+48q+624q^2+1344q^3+5232q^4 + O(q^5)$.

$\ds E_{4,2}^{(1)}(\tau) = \Delta_2(\tau) = \frac{1}{256} {\twpa}(\frac{1}{2},\tau)^2 = q+8q^2+28q^3+64q^4+ O(q^{5})$
\ms

This basis provides the generator of the first non-trivial cupsidal space, $S_8(\Gamma_0(2))$:

\[
\begin{array}{lcl}
F_{8,2}^{(1)} & = & [E_{4,2}^{(0)} - 64E_{4,2}^{(1)}]E_{4,2}^{(1)}\\
& = & \ds q \prod_{k=1}^{+\infty}(1-q^k)^8(1-q^{2k})^8\\
& = & q - 8q^2 + 12q^3 + 64q^4 - 210q^5 - 96q^6 + 1016q^7 + O(q^8))
\end{array}
\]
\bs

We also have the representation $F_{8,2}^{(1)} = \frac{1}{256}{\twpa}(\frac{1}{2},\tau)^2 {\twpa}(\tau,2\tau)^2$, which justifies the previous factorized form.
Indeed, both components of the product are in $M_4(\Gamma_0(2))$, cusps behavior is easily verified with Weierstrass representation. As a $\eta$-product, this result is well known elsewhere, see \cite{Kohl} and \cite{Ono} for example.
\bs

For $k\in \llbracket 5,10\rrbracket$, $\dim(S_{2k}(\Gamma_0(2))) - \dim(M_{2(k-4)}(\Gamma_0(2))) = 0 = \dim(S_{8}(\Gamma_0(2)))-1$.

In application of lemma III-\ref{LemP}, for $k_0 = 4$, the generic form of a $S_{2k}(\Gamma_0(2))$ unitary upper triangular basis ${\cal C}_{2k}(\Gamma_0(2))$ is:

\begin{equation*}
\forall k\se 4, \ \ {\cal C}_{2k}(\Gamma_0(2)) = F_{8,2}^{(1)}{\cal B}_{2k-8}(\Gamma_0(2)).
\end{equation*}
\bs

\subsection{-- Structure and bases of $(S_{2k}(\Gamma_0(3)))_{k\in \N^*}$}
\ms

The modular form $\Delta_3\in M_{6}(\Gamma_0(3))$ that structures this set of spaces is:

\begin{center}
\begin{tabular}{lcl}
$\Delta_3(\tau)$ & $=$ & $\ds \eta(3\tau)^{18} \eta(\tau)^{-6}$\\
  & $=$ & $\ds {q}^{2}\prod _{k=1}^{+\infty} \frac{(1-{q}^{3k})^{18}}{(1-{q}^{k})^{6}}$\\
  & $=$ & $\ds  {q}^{2}+6\,{q}^{3}+27\,{q}^{4}+80\,{q}^{5}+207\,{q}^{6}+432\,{q}^{7}+
863\,{q}^{8}+1512\,{q}^{9}+O(q^{10})$\\
\end{tabular}
\end{center}
\ms

For $k\in \N^*$, let us note $(F_{2k,3}^{(r)})_{1\ie r \ie \delta_{2k}(\Gamma_0(3))}$ a $S_{2k}(\Gamma_0(3))$ unitary upper triangular basis. Let us remind the first values of $\delta_{2k}(\Gamma_0(3))$:

\begin{center}
\begin{tabular}{ c|c|c|c|c|c|c|c|c|c|c}
$2k$   & 2  & 4  & 6  & 8  & 10  & 12  & 14  & 16 & 18 & 20\\
\hline
$\delta_{2k}(\Gamma_0(3))$ & 0  & 0  & 1  & 1  & 2   &  3  & 3   & 4 & 5 & 5 \\
\end{tabular}
\end{center}

The unitary generator of $S_{6}(\Gamma_0(3))$ must be in $M_{6}(\Gamma_0(3))$ and thanks to the asymptotic expansion of $F_{6,3}^{(1)}$ given by SAGE, we get:
\be
\item[]
\begin{tabular}{lcl}
$F_{6,3}^{(1)}(\tau) $ & = & $E_{6,3}^{(1)} - 27E_{6,3}^{(2)}$.
\end{tabular}
\ee

For $k\in \llbracket 4,9\rrbracket$, $\dim(S_{2k}(\Gamma_0(3))) - \dim(M_{2(k-3)}(\Gamma_0(3))) = 0 = \dim(S_{6}(\Gamma_0(3)))-1$.

In application of lemma III-\ref{LemP}, for $k_0 = $3, the generic form of a $S_{2k}(\Gamma_0(3))$ unitary upper triangular basis ${\cal C}_{2k}(\Gamma_0(3))$ is:
\begin{equation*}
\forall k\se 3, \ \ {\cal C}_{2k}(\Gamma_0(3)) = F_{6,3}^{(1)}{\cal B}_{2k-6}(\Gamma_0(3)).
\end{equation*}
\bs
\ms

\subsection{-- Structure and bases of $(S_{2k}(\Gamma_0(4)))_{k\in \N^*}$}
\ms

The modular form that structures this set of spaces is $\Delta_4$ defined as follows:

\begin{equation*}
\Delta_4(\tau) = \ds \eta(4\tau)^{8} \eta(2\tau)^{-4} = \ds -\frac{1}{16} {\twpa}(\frac{1}{2},2\tau)\in M_{2}(\Gamma_0(4)).
\end{equation*}

The first dimension table of $(S_{2k}(\Gamma_0(4)))_{k\in \N^*}$:
\ms

\begin{center}
\begin{tabular}{ c|c|c|c|c|c|c|c|c}
$2k$   & 2  & 4  & 6  & 8  & 10  & 12  & 14  & 16 \\
\hline
$\delta_{2k}(\Gamma_0(4))$    & 0  & 0  & 1  & 2  & 3  & 4  & 5  & 6 \\
\end{tabular}
\end{center}

SAGE indicates $F_{4,6}^{(1)}(\tau) = q-12\,{q}^{3}+54\,{q}^{5}-88\,{q}^{7}+O \left( {q}^{8} \right)$.
\ms

One poses

\begin{tabular}{lcl}
$E_{2,4}^{(0)}$ & = & $\ds {\twpa}(\tau,2\tau)$\\
								& = & $\ds \prod_{k=1}^{+\infty} \frac{(1-{q}^{n})^{8}}{(1-{q}^{2n})^{4}}$\\
                & = & $1-8q+24q^2+32q^3+24q^4-48q^5+96q^6+O(q^7).$
\end{tabular}

\begin{tabular}{lcl}
$E_{2,4}^{(1)}$ & = & $\Delta_4 = \ds -\frac{1}{16} {\twpa}(\frac{1}{2},2\tau)$\\
  & = & $\ds q\prod_{k=1}^{+\infty} \frac{(1-{q}^{4n})^{8}}{(1-{q}^{2n})^{4}}$\\
                & = & $\ds q+4q^3+6q^5+8q^7+13q^9+O(q^{11}).$
\end{tabular}
\ms

A $M_{2k}(\Gamma_0(4))$ unitary upper triangular basis is given by:

\[{\cal B}_{2k}(\Gamma_0(4)) = \pa{[E_{2,4}^{(0)}]^a [E_{2,4}^{(1)}]^b, \ {\rm with} \ (a,b)\in\N^2 \ {\rm such \ that} \ a+b = k}.\]

Thus ${\cal B}_{6}(\Gamma_0(4)) = (E_{4,6}^{(0)}, \ E_{4,6}^{(1)}, \ E_{4,6}^{(2)}, \ E_{4,6}^{(3)}) = ([E_{2,4}^{(0)}]^3, \ [E_{2,4}^{(0)}]^2[E_{2,4}^{(0)}], \ [E_{2,4}^{(0)}][E_{2,4}^{(0)}]^2, \ [E_{2,4}^{(0)}]^3)$ which identifies the unitary generator of $S_{6}(\Gamma_0(4))$:

\begin{tabular}{lcl}
$F_{6,4}^{(1)}$ & = & $\ds E_{6,4}^{(1)} + 16E_{6,4}^{(2)}$\\
  & = & $\ds E_{2,4}^{(0)}E_{2,4}^{(1)}[E_{2,4}^{(0)}+16E_{2,4}^{(1)}]$\\
                & = & $\ds q-12q^3+54q^5-88q^7-99q^9+O(q^{11}).$
\end{tabular}
\ms

For $k\in \llbracket 4,9 \rrbracket$, $\dim(S_{2k}(\Gamma_0(4))) - \dim(M_{2(k-3)}(\Gamma_0(4))) = 0 = \dim(S_{6}(\Gamma_0(4)))-1$.

In application of lemma III-\ref{LemP}, for $k_0 = $3, the generic form of a $S_{2k}(\Gamma_0(4))$ unitary upper triangular basis ${\cal C}_{2k}(\Gamma_0(4))$ is:

\begin{equation*}
\begin{array}{lcl}
\forall k \se 3, \ \ {\cal C}_{2k}(\Gamma_0(4)) & = & F_{6,4}^{(1)}{\cal B}_{2k-6}(\Gamma_0(4))
\end{array}
\end{equation*}

\subsection{-- Structure and bases of $(M_{2k}(\Gamma_0(5)))_{k\in \N^*}$}
\ms

The strong modular unit that structures the modular spaces is:
\[\Delta_5(\tau) = \eta(5\tau)^{10}\eta(\tau)^{-2} = \frac{1}{16} ({\wpa}(\tau,5\tau) - {\wpa}(2\tau,5\tau))^2 = {q}^{2}\prod _{n=1}^{+\infty}{\frac { \left( 1-{q}^{5n}
 \right) ^{10}}{ \left( 1-{q}^{n} \right) ^{2}}}\in M_{4}(\Gamma_0(5)).\]

The table of the dimensions of the first spaces:

\begin{center}
\begin{tabular}{ c|c|c|c|c|c|c|c|c}
$2k$   & 2  & 4  & 6  & 8  & 10  & 12  & 14  & 16 \\
\hline
$\delta_{2k}(\Gamma_0(5))$    & 0  & 1  & 1  & 3  & 3  & 5  & 5  & 7 \\
\end{tabular}
\end{center}

We remind the functions generating first spaces $M_{2k}(\Gamma_0(5))$.
\ms

The standard $M_{2}(\Gamma_0(5))$ generator

\begin{tabular}{lcl}
$E_{2,5}^{(0)}$ & = & $\ds -\frac{3}{2} ({\wpa}(\tau,5\tau) + {\wpa}(2\tau,5\tau))$
\end{tabular}
\ms

A $M_{4}(\Gamma_0(5))$ unitary upper triangular basis
\ms

\begin{tabular}{lclcl}
$E_{4,5}^{(0)}(\tau)$ & = & $[E_{2,5}^{(0)}]^2$ & = & $\ds \frac{9}{4} \pa{{\wpa}(\tau,5\tau) + {\wpa}(2\tau,5\tau)}^2$
\end{tabular}
\ms

\begin{tabular}{lcl}
$E_{4,5}^{(1)}(\tau)$ & = & $\ds \frac{1}{48} \pa{9 \pa{{\wpa}(\tau,5\tau) + {\wpa}(2\tau,5\tau)}^2 -12 \pa{{\wpa}(\frac{1}{2},5\tau)^2 + {\wpa}(\frac{5\tau}{2},5\tau)^2 + {\wpa}(\frac{1}{2},5\tau) {\wpa}(\frac{5\tau}{2},5\tau)}}$
\end{tabular}
\ms

\begin{tabular}{lclcl}
$E_{4,5}^{(2)}(\tau)$ & = & $\Delta_5(\tau)$
 & = & $\ds \frac{1}{16} ({\wpa}(\tau,5\tau) - {\wpa}(2\tau,5\tau))^2$
\end{tabular}
\ms

The asymptotic expansion of the unitary generator of $S_{4}(\Gamma_0(5))$ given by SAGE ($F_{4,5}^{(1)} = q - 4q^2 + O(q^3)$) enables one to identify thanks to the previous basis:
\[F_{4,5}^{(1)} = E_{4,5}^{(1)} - 10 E_{4,5}^{(2)}.\]

For $k\in \llbracket 3,8\rrbracket$, $\dim(S_{2k}(\Gamma_0(5))) - \dim(M_{2(k-2)}(\Gamma_0(5))) = 0 = \dim(S_{4}(\Gamma_0(5)))-1$.

In application of lemma III-\ref{LemP}, for $k_0 = 2$, the generic form of a $S_{2k}(\Gamma_0(5))$ unitary upper triangular basis ${\cal C}_{2k}(\Gamma_0(5))$ is:
\begin{equation*}
\forall k\se 2, \ \ {\cal C}_{2k}(\Gamma_0(5)) = F_{4,5}^{(1)}{\cal B}_{2k-4}(\Gamma_0(5)).
\end{equation*}
\bs
\ms

\subsection{-- Structure and bases of $(M_{2k}(\Gamma_0(6)))_{k\in \N^*}$}
\ms


We know that
\begin{equation}
\begin{array}{lcl}
\Delta_6(\tau) &  = & \ds \frac{\eta(\tau)^2 \eta(6\tau)^{12}}{\eta(2\tau)^4 \eta(3\tau)^6} \in M_{2}(\Gamma_0(6))\\
& = & \ds q^2\prod_{k=1}^{+\infty} \frac{(1-q^{k})^{2}(1-q^{6k})^{12}} {(1-q^{2k})^{4} (1-q^{3k})^{6}}
\end{array}
\end{equation}
\ms

The dimensions of the first spaces $(M_{2k}(\Gamma_0(6)))_{k\in \N^*}$ are shown in the following table:

\begin{center}
\begin{tabular}{ c|c|c|c|c|c|c|c|c}
$2k$   & 2  & 4  & 6  & 8  & 10  & 12  & 14  & 16 \\
\hline
$\delta_{2k}(\Gamma_0(6))$   & 0  & 1  & 3  & 5  & 7  & 9  & 11  & 13 \\
\end{tabular}
\end{center}
\bs

We remind the $M_{2}(\Gamma_0(6))$ unitary upper triangular basis obtained in Part II:
\ms

\begin{tabular}{lcl}
$E_{2,6}^{(0)}(\tau)$ & = & $\ds -3{\wpa}(\tau,2\tau)$\\
                & = & $1+24q+24q^2+96q^3+24q^4+144q^5+96q^6+192*q^7+O(q^8)$
\end{tabular}

\begin{tabular}{lcl}
$E_{2,6}^{(1)}(\tau)$ & = & $\ds -\frac{1}{4}\pa{{\wpa}(\tau,2\tau) - {\wpa}(\tau,3\tau)}$\\
& = & $\ds q-q^2+7q^3-5q^4+6q^5+5q^6+8q^7 + O(q^8)$
\end{tabular}

\begin{tabular}{lcl}
$E_{2,6}^{(2)}(\tau)$ & = & $\Delta_6(\tau) = \ds q^2\prod_{k=1}^{+\infty} \frac{(1-q^{k})^{2}(1-q^{6k})^{12}} {(1-q^{2k})^{4} (1-q^{3k})^{6}} =  \frac{\eta(\tau)^2 \eta(6\tau)^{12}}{\eta(2\tau)^4 \eta(3\tau)^6}$\\
                &  = & $\ds \frac{1}{48} \pa{3\wpa(\tau,2\tau) - 8\wpa(\tau,3\tau) + \sum_{k=1}^5 {\wpa}(k\tau,6\tau)}$\\
                & = & $q^2-2q^3+3q^4-q^6+7q^8-8q^9+6q^{10}+O(q^{11})$
\end{tabular}
\bs

The asymptotic expansion of the $S_{4}(\Gamma_0(6))$ unitary generator enables one to identify thanks to the previous basis:
\[F_{4,6}^{(1)} = E_{4,6}^{(1)} - 25 E_{4,6}^{(2)} + 540 E_{4,6}^{(3)} + 864 E_{4,6}^{(4)}.\]
\ms

For $k\in \llbracket 3,8\rrbracket$, $\dim(S_{2k}(\Gamma_0(6))) - \dim(M_{2(k-2)}(\Gamma_0(6))) = 0 = \dim(S_{4}(\Gamma_0(6)))-1$.

In application of lemma III-\ref{LemP}, for $k_0 = 2$, the generic form of a $S_{2k}(\Gamma_0(6))$ unitary upper triangular basis ${\cal C}_{2k}(\Gamma_0(6))$ is:

\[\forall k\se 2, \ \ {\cal C}_{2k}(\Gamma_0(6)) = F_{4,6}^{(1)}{\cal B}_{2k-4}(\Gamma_0(6))\]

\bs
\ms

\subsection{-- Structure and bases of $(M_{2k}(\Gamma_0(7)))_{k\in \N^*}$}
\ms

The strong modular unit of level $7$ is:
\[\Delta_7(\tau) = \eta(\tau)^{-2}\eta(7\tau)^{14} = q^4 \prod_{n=1}^{+\infty} \frac{(1-q^{7n})^{14}}{(1-q^n)^2} \in M_6(\Gamma_0(7)).\]
\ms

The dimensions of the first spaces $(M_{2k}(\Gamma_0(7)))_{k\in \N^*}$ are given in the following table:

\begin{center}
\begin{tabular}{ c|c|c|c|c|c|c|c|c}
$2k$   & 2  & 4  & 6  & 8  & 10  & 12  & 14  & 16 \\
\hline
$\delta_{2k}(\Gamma_0(7))$   & 0  & 1  & 3  & 3  & 5  & 7  & 7  & 9 \\
\end{tabular}
\end{center}
\bs

A generator of $M_{2}(\Gamma_0(7))$ is given by the generic formula:

\begin{tabular}{lcl}
$E_{2,7}^{(0)}(\tau)$ & = & $\ds - \sum_{k=1}^3 {\wpa}(k\tau,7\tau)$
\end{tabular}
\ms

We remind a $M_{4}(\Gamma_0(7))$ unitary upper triangular basis:

\begin{tabular}{lcl}
$E_{4,7}^{(0)}(\tau)$ & = & $\ds [E_{2,7}^{(0)}]^2 = \pa{\sum_{k=1}^3 {\wpa}(k\tau,7\tau)}^2$
\end{tabular}

\begin{tabular}{lcl}
$E_{4,7}^{(1)}(\tau)$ & = & $\ds \frac{1}{8} \pa{\pa{\sum_{k=1}^3 {\wpa}(k\tau,7\tau)}^2 - 3 \pa{{\wpa}(\frac{1}{2},7\tau)^2 + {\wpa}(\frac{7\tau}{2},7\tau)^2 + {\wpa}(\frac{1}{2},7\tau) {\wpa}(\frac{7\tau}{2},7\tau)}}$
\end{tabular}

\begin{tabular}{lcl}
$E_{4,7}^{(2)}(\tau)$ & = & $\ds \frac{1}{32} \pa{3\sum_{k=1}^3 {\wpa}(k\tau,7\tau)^2 - \pa{\sum_{k=1}^3 {\wpa}(k\tau,7\tau)}^2}$
\end{tabular}
\ms

The asymptotic expansion of the $S_{4}(\Gamma_0(7))$ unitary generator enables one to identify thanks to the previous database:
\[F_{4,7}^{(1)} = E_{4,7}^{(1)} - 6 E_{4,7}^{(2)}.\]

We could hope to apply lemma III-\ref{LemP} as before. But if the equality ${\cal C}_{2k}(\Gamma_0(7)) = F_{4,7}^{(1)}{\cal B}_{2k-4}(\Gamma_0(7))$ is true when $k$ is not a multiple of $3$, it becomes false when $3$ divides $k$: the space  dimensions indicates that two items are missing from the ${\cal C}_{2k}(\Gamma_0(7))$ basis.
\ms

We must then go back to Theorem I-8.2. which requires knowing a $S_{6}(\Gamma_0(7))$ basis, and for this we remind the $M_{6}(\Gamma_0(7))$ unitary upper triangular basis found in Part II.
\ms

\begin{tabular}{lcl}
$E_{6,7}^{(0)}(\tau)$ & = & $\ds [E_{2,7}^{(0)}]^3$\\
                & = & $1+12q+84q^2+400q^3+1476q^4+4392q^5+11184q^6+24780q^7+49668q^8+O(q^9)$
\end{tabular}
\ms

\begin{tabular}{lcl}
$E_{6,7}^{(1)}(\tau)$ & = & $\ds E_{2,7}^{(0)}E_{4,7}^{(1)}$\\
                & = & $q+9*q^2+48*q^3+181*q^4+546*q^5+1392*q^6+3067*q^7+6081*q^8+O(q^9)$
\end{tabular}
\ms

\begin{tabular}{lcl}
$E_{6,7}^{(2)}(\tau)$ & = & $\ds E_{2,7}^{(0)}E_{4,7}^{(2)}$\\
                & = & $q^2+7q^3+32q^4+95q^5+241q^6+503q^7+1017q^8+O(q^9)$
\end{tabular}
\ms

\begin{tabular}{lcl}
$E_{6,7}^{(3)}(\tau)$ & = & $\ds -\frac{1}{128} (2\wpa(\tau,7\tau) - \wpa(2\tau,7\tau) - \wpa(3\tau,7\tau))(2\wpa(2\tau,7\tau) - \wpa(\tau,7\tau) - \wpa(3\tau,7\tau))$\\
& & {\hskip 7cm} $(2\wpa(3\tau,7\tau) - \wpa(\tau,7\tau) - \wpa(2\tau,7\tau))$
\end{tabular}
\ms

\begin{tabular}{lcl}
$E_{6,7}^{(4)}(\tau)$ & = & $\ds \Delta_7(\tau) = \ds q^4 \prod_{n=1}^{+\infty} \frac{(1-q^{7n})^{14}}{(1-q^n)^2}$
\end{tabular}
\ms

Asymptotic expansions from a $S_{4}(\Gamma_0(7))$ basis found by SAGE, and the $F_{6,7}^{(1)}$ element easily accessible lead to the following basis:

\[F_{6,7}^{(1)} = F_{4,7}^{(1)} E_{2,7}^{(0)} = E_{6,7}^{(1)} - 6E_{2,7}^{(2)}\]

\[F_{6,7}^{(2)} = E_{6,7}^{(2)} - 49E_{2,7}^{(4)}\]

\[F_{6,7}^{(3)} = E_{6,7}^{(3)} - \frac{13}{2}E_{2,7}^{(4)}\]

\ms

For $k\in \llbracket 4,9\rrbracket$, $\dim(S_{2k}(\Gamma_0(7))) - \dim(M_{2(k-2)}(\Gamma_0(7))) = 2 = \dim(S_{4}(\Gamma_0(7)))-1$.

In application of lemma III-\ref{LemP}, for $k_0 = $3, the generic form of a $S_{2k}(\Gamma_0(7))$ unitary upper triangular basis ${\cal C}_{2k}(\Gamma_0(7))$ is:

\[\forall k\se 3, \ \ {\cal C}_{2k}(\Gamma_0(7)) = (F_{6,7}^{(1)}[E_{2,7}^{(0)}]^{k-3}, \ F_{6,7}^{(2)}[E_{2,7}^{(0)}]^{k-3}) \cup F_{6,7}^{(3)}{\cal B}_{2k-6}(\Gamma_0(7)).\] 

Finally, for $k\se 1$, thanks to arguments similar to those developed in lemma III-\ref{LemP}, we can also give the generic form of a $S_{2k}(\Gamma_0(7))$ unitary upper triangular basis ${\cal C}_{6k+2}(\Gamma_0(7))$ in the form:

\[{\cal C}_{6k+2}(\Gamma_0(7)) = F_{4,7}^{(1)}{\cal B}_{6k-2}(\Gamma_0(7))\]

\[{\cal C}_{6k+4}(\Gamma_0(7)) = F_{4,7}^{(1)}{\cal B}_{6k}(\Gamma_0(7))\]

\[{\cal C}_{6k+6}(\Gamma_0(7)) = (F_{6,7}^{(1)}[E_{2,7}^{(0)}]^{3k}, \ F_{6,7}^{(2)}[E_{2,7}^{(0)}]^{3k}) \cup F_{6,7}^{(3)}{\cal B}_{6k}(\Gamma_0(7)).\]

\bs
\ms

\subsection{-- Structure and bases of $(M_{2k}(\Gamma_0(8)))_{k\in \N^*}$}
\ms

The dimensions of the first spaces $(S_{2k}(\Gamma_0(8)))_{k\in \N^*}$ are as follows:

\begin{center}
\begin{tabular}{ c|c|c|c|c|c|c|c|c}
$2k$   & 2  & 4  & 6  & 8  & 10  & 12  & 14  & 16 \\
\hline
$\delta_{2k}(\Gamma_0(8))$   & 0  & 1  & 3  & 5  & 7  & 9  & 11  & 13 \\
\end{tabular}
\end{center}

The structure here is very simple since $\Delta_8(\tau) = \Delta_4(2\tau)\in M_{2}(\Gamma_0(8))$.
\ms

A $M_{2}(\Gamma_0(8))$ unitary upper triangular basis consisting of modular units (not strong, except one):
\ms

\begin{tabular}{lcl}
$E_{2,8}^{(0)}(\tau)$ & = &  $\ds {\twpa}(\tau,2\tau)$\\
& = & $\ds \eta(\tau)^8\eta(2\tau)^{-4} = \prod_{k=1}^{+\infty} \frac{(1-{q}^{n})^{8}}{(1-{q}^{2n})^{4}}$\\
& = & $1-8q+24q^2-32q^3+24q^4-48q^5+96q^6-64q^7+24q^8+O(q^9)$.
\end{tabular}
\ms

\begin{tabular}{lcl}
$E_{2,8}^{(1)}(\tau)$ & = & $\ds -\frac{1}{16} {\twpa}(\frac{1}{2},2\tau)$\\
								& = & $\ds \Delta_4(\tau) = \eta(2\tau)^{-4} \eta(4\tau)^{8} = q\prod_{k=1}^{+\infty} \frac{(1-{q}^{4n})^{8}}{(1-{q}^{2n})^{4}}$\\
                & = & $\ds q+4q^3+6q^5+8q^7+13q^9+O(q^{10})$.
\end{tabular}

\begin{tabular}{lcl}
$E_{2,8}^{(2)}(\tau)$ & = & $\ds -\frac{1}{16} {\twpa}(\frac{1}{2},4\tau)$\\
                & = & $ \ds \Delta_8(\tau) = \eta(4\tau)^{-4} \eta(8\tau)^{8} = q^2\prod_{k=1}^{+\infty} \frac{(1-{q}^{8n})^{8}}{(1-{q}^{4n})^{4}}$\\
                & = & $q^2+4q^6+O(q^{10})$.
\end{tabular}
\ms

We also need a $M_{4}(\Gamma_0(8))$ basis:
\[{\cal B}_4(\Gamma_0(8)) = (E_{4,8}^{(0)}, \ E_{4,8}^{(1)}, \ E_{4,8}^{(2)}, \ E_{4,8}^{(3)}, \ E_{4,8}^{(4)}) = \pa{[E_{2,8}^{(0)}]^2, \ E_{2,8}^{(0)}E_{2,8}^{(1)}, \ E_{2,8}^{(0)}E_{2,8}^{(2)}, \ E_{2,8}^{(1)}E_{2,8}^{(2)}, \ E_{2,8}^{(2)}E_{2,8}^{(2)}}.\]

Still by an asymptotic expansion, we obtain the unitary generator of $S_{4}(\Gamma_0(8))$:
\[F_{4,8}^{(1)} = E_{4,8}^{(1)} + 8E_{4,8}^{(2)} + 32E_{4,8}^{(3)} - 128E_{4,8}^{(4)}.\]

For $k\in \llbracket 3,8\rrbracket$, $\dim(S_{2k}(\Gamma_0(8))) - \dim(M_{2(k-2)}(\Gamma_0(8))) = 0 = \dim(S_{4}(\Gamma_0(8)))-1$.

In application of lemma III-\ref{LemP}, for $k_0 = 2$, the generic form of a $S_{2k}(\Gamma_0(8))$ unitary upper triangular basis ${\cal C}_{2k}(\Gamma_0(8))$ is:
\[\forall k \se 2, \ \ {\cal C}_{2k}(\Gamma_0(8)) = F_{4,8}^{(1)}{\cal B}_{2k-4}(\Gamma_0(8)).\]

\bs
\ms

\subsection{-- Structure and bases of $(M_{2k}(\Gamma_0(9)))_{k\in \N^*}$}
\ms

Let us remind the first values of $\delta_{2k}(\Gamma_0(9))$:

\begin{center}
\begin{tabular}{ c|c|c|c|c|c|c|c|c}
$2k$   & 2  & 4  & 6  & 8  & 10  & 12  & 14  & 16 \\
\hline
$\delta_{2k}(\Gamma_0(9))$ & 0  & 1  & 3  & 5  & 7   &  9  & 11   & 13  \\
\end{tabular}
\end{center}

A $M_{2}(\Gamma_0(9))$ unitary upper triangular basis:
\ms

\begin{tabular}{lcl}
$E_{2,9}^{(0)}$ & = & $\ds 3 {\wpa}(3\tau,9\tau)$\\
                & = & $\ds 1+12q^3+36q^6+12q^9+84q^{12}+72q^{15}+O(q^{18})$
\end{tabular}
\ms

\begin{tabular}{lcl}
$E_{2,9}^{(1)}$ & = & $-\ds \frac{1}{4} \pa{{\wpa}(\tau,3\tau)-{\wpa}(3\tau,9\tau)}$\\
                & = & $\ds q+3q^2+7q^4+6q^5+8q^7+15q^8+O(q^{10})$
\end{tabular}

\begin{tabular}{lcl}
$E_{2,9}^{(2)}$ & = & $\Delta_9(\tau) = \ds \eta(9\tau)^{6} \eta(3\tau)^{-2}$\\
  & $=$ & $\ds {q}^{2}\prod_{k=1}^{+\infty}{\frac { \left( 1-{q}^{9k} \right) ^{6}}{ \left( 1-{q}^{3k} \right) ^{2}}}$\\
  & $=$ & $\ds q^2+2q^5+5q^8+4q^{11}+8q^{14} + O(q^{17})$
\end{tabular}
\ms

We also need a $M_{4}(\Gamma_0(9))$ unitary upper triangular basis given by:
\[{\cal B}_4(\Gamma_0(9)) = (E_{4,9}^{(0)}, \ E_{4,9}^{(1)}, \ E_{4,9}^{(2)}, \ E_{4,9}^{(3)}, \ E_{4,9}^{(4)}) = \pa{[E_{2,9}^{(0)}]^2, \ E_{2,9}^{(0)}E_{2,9}^{(1)}, \ E_{2,9}^{(0)}E_{2,9}^{(2)}, \ E_{2,9}^{(1)}E_{2,9}^{(2)}, \ E_{2,9}^{(2)}E_{2,9}^{(2)}}.\]

Still by an asymptotic expansion, we obtain the unitary generator of $S_{4}(\Gamma_0(9))$:
\[F_{4,9}^{(1)} = E_{4,9}^{(1)} - 3E_{4,8}^{(2)} - 27E_{4,8}^{(4)}.\]

For $k\in \llbracket 3,8\rrbracket$, $\dim(S_{2k}(\Gamma_0(9))) - \dim(M_{2(k-2)}(\Gamma_0(9))) = 0 = \dim(S_{4}(\Gamma_0(9)))-1$.

In application of lemma III-\ref{LemP}, for $k_0 = 2$, the generic form of a $S_{2k}(\Gamma_0(9))$ unitary upper triangular basis ${\cal C}_{2k}(\Gamma_0(9))$ is:

\[\forall k \se 2, \ \ {\cal C}_{2k}(\Gamma_0(9)) = F_{4,9}^{(1)}{\cal B}_{2k-4}(\Gamma_0(9)).\]

\bs
\ms

\subsection{-- Structure and bases of $(M_{2k}(\Gamma_0(10)))_{k\in \N^*}$}
\ms

Let us remind the first values of $\delta_{2k}(\Gamma_0(10))$:

\begin{center}
\begin{tabular}{ c|c|c|c|c|c|c|c|c}
$2k$   & 2  & 4  & 6  & 8  & 10  & 12  & 14  & 16 \\
\hline
$\delta_{2k}(\Gamma_0(10))$ & 0  & 3  & 5  & 9  & 11   &  15  & 17  & 21  \\
\end{tabular}
\end{center}

We get the following $M_{2}(\Gamma_0(10))$ unitary upper triangular basis.
\ms

\begin{tabular}{lcl}
$E_{2,10}^{(0)}$ & = & $-\ds 3 {\wpa}(5\tau,10\tau)$
\end{tabular}
\ms

\begin{tabular}{lcl}
$E_{2,10}^{(1)}$ & = & $-\ds \frac{1}{8}\pa{{\wpa}(\tau,2\tau)-{\wpa}(5\tau,10\tau)}$
\end{tabular}
\ms

\begin{tabular}{lcl}
$E_{2,10}^{(2)}$ & = & $\ds \frac{1}{16}\pa{{\wpa}(\tau,2\tau)-2{\wpa}(\tau,5\tau)-2{\wpa}(2\tau,5\tau)+3{\wpa}(5\tau,10\tau)}$
\end{tabular}
\bs

We also need a $M_{4}(\Gamma_0(10))$ unitary upper triangular basis.

\begin{tabular}{lcl}
$E_{4,10}^{(0)}$ & = & $[E_{2,10}^{(0)}]^2 $
\end{tabular}

\begin{tabular}{lcl}
$E_{4,10}^{(1)}$ & = & $E_{2,10}^{(0)} E_{2,10}^{(1)} $
\end{tabular}

\begin{tabular}{lcl}
$E_{4,10}^{(2)}$ & = & $ E_{2,10}^{(0)}E_{2,10}^{(2)} $
\end{tabular}

\begin{tabular}{lcl}
$E_{4,10}^{(3)}$ & = & $ E_{2,10}^{(1)}E_{2,10}^{(2)} $
\end{tabular}

\begin{tabular}{lcl}
$E_{4,10}^{(4)}$ & = & $ [E_{2,10}^{(2)}]^2$
\end{tabular}

\begin{tabular}{lcl}
$E_{4,10}^{(5)}$ & = & $\Delta_2(5\tau) = \ds \frac{1}{256} {\twpa}(\frac{1}{2},5\tau)^2$
\end{tabular}

\begin{tabular}{lcl}
$E_{4,10}^{(6)}$ & = & $\Delta_{10}(\tau) = \ds \eta(\tau)^2 \eta(2\tau)^{-4} \eta(5\tau)^{-10} \eta(10\tau)^{20}$
\end{tabular}
\bs

We then obtain thanks to the identification of the asymptotic expansions:

\[F_{4,10}^{(1)} = E_{4,10}^{(1)} - E_{4,10}^{(2)} - 4E_{4,10}^{(3)} + 2E_{4,10}^{(4)} + 16E_{4,10}^{(5)} - 40E_{4,10}^{(6)}\]

\[F_{4,10}^{(2)} = E_{4,10}^{(2)} - 7E_{4,10}^{(4)} + 4E_{4,10}^{(5)} +  40E_{4,10}^{(6)}\]

\[F_{4,10}^{(3)} = E_{4,10}^{(3)} - 3E_{4,10}^{(4)} - 8E_{4,10}^{(5)} +  20E_{4,10}^{(6)}\]

For $k\in \llbracket 3,8\rrbracket$, $\dim(S_{2k}(\Gamma_0(10))) - \dim(M_{2(k-2)}(\Gamma_0(10))) = 2 = \dim(S_{4}(\Gamma_0(10)))-1$.

In application of lemma III-\ref{LemP}, for $k_0 = 2$, the generic form of a $S_{2k}(\Gamma_0(10))$ unitary upper triangular basis ${\cal C}_{2k}(\Gamma_0(10))$ is:

\[\forall k\se 2, \ \ {\cal C}_{2k}(\Gamma_0(10)) = (F_{4,10}^{(1)}[E_{2,10}^{(0)}]^{k-2}, \ F_{4,10}^{(2)}[E_{2,10}^{(0)}]^{k-2}) \cup F_{4,10}^{(3)}{\cal B}_{2k-4}(\Gamma_0(10)).\]

\bs
\bs
\bs
\bs

\end{document}